\documentclass[10pt, reqno]{amsart}
\usepackage{graphicx, amssymb, amsmath, amsthm, color, slashed}
\numberwithin{equation}{section}

\usepackage{hyperref}

\let\Re=\undefined\DeclareMathOperator*{\Re}{Re}
\let\Im=\undefined\DeclareMathOperator*{\Im}{Im}

\newcommand{\R}{\mathbb{R}}
\newcommand{\C}{\mathbb{C}}

\newcommand{\eps}{\varepsilon}

\newcommand{\Z}{\mathbb{Z}}

\newtheorem{theorem}{Theorem}[section]

\newtheorem{lemma}[theorem]{Lemma}

\theoremstyle{definition}

\theoremstyle{remark}

\newcommand{\qtq}[1]{\quad\text{#1}\quad}

\begin{document}
\title[Focusing NLS]{A new proof of scattering below the ground state for the non-radial focusing NLS}
\author[B. Dodson]{Benjamin Dodson}
\address{Department of Mathematics, Johns Hopkins University}
\email{bdodson4@jhu.edu}
\author[J. Murphy]{Jason Murphy}
\address{Department of Mathematics and Statistics, Missouri University of Science and Technology}
\email{jason.murphy@mst.edu}

\begin{abstract} We revisit the scattering result of Duyckaerts, Holmer, and Roudenko for the non-radial $\dot H^{1/2}$-critical focusing NLS.  By proving an interaction Morawetz inequality, we give a simple proof of scattering below the ground state in dimensions $d\geq 3$ that avoids the use of concentration compactness. 
\end{abstract}

\maketitle

\section{Introduction} 

We consider the initial-value problem for the focusing $\dot H^{\frac12}$-critical nonlinear Schr\"odinger equation (NLS) in dimensions $d\geq 3:$
\begin{equation}\label{nls}
\begin{cases}
(i\partial_t+\Delta) u = -|u|^{\frac{4}{d-1}}u \\
u(0)=u_0\in H^1(\R^d),
\end{cases}
\end{equation}
where $u:\R\times\R^d\to\C$.  This includes the $3d$ cubic NLS, which we studied in our previous work \cite{DM} in the radial setting.  In this work, we extend our arguments to address the non-radial case. 

Solutions to \eqref{nls} conserve \emph{mass} and \emph{energy}, defined respectively by
\begin{align*}
M(u(t)) &= \int_{\R^d} |u(t,x)|^2\,dx, \\
E(u(t)) &= \int_{\R^d} \tfrac12 |\nabla u(t,x)|^2 - \tfrac{d-1}{2(d+1)}|u(t,x)|^\frac{2(d+1)}{d-1}\,dx.
\end{align*}
The equation \eqref{nls} is $\dot H^{1/2}$-critical in the sense that the $\dot H^{1/2}$-norm of the initial data is invariant under the scaling that preserves the class of solutions, namely,
\begin{equation}\label{scale}
u(t,x)\mapsto \lambda^{\frac{d-1}{2}} u(\lambda^2 t,\lambda x).
\end{equation}

By solution, we mean a function $u\in C_t H_x^1(I\times\R^d)$ on an interval $I\ni 0$ satisfying the Duhamel formula
\[
u(t)=e^{it\Delta}u_0+i\int_0^t e^{i(t-s)\Delta}(|u|^{\frac{4}{d-1}} u)(s)\,ds
\]
for $t\in I$, where $e^{it\Delta}$ is the Schr\"odinger group.  We write $I_{\max}$ for the maximal-lifespan of $u$ and call $u$ global if $I_{\max}=\R$.  A global solution $u$ \emph{scatters} if there exist $u_\pm\in H^1(\R^d)$ so that
\[
\lim_{t\to\pm\infty}\|u(t)-e^{it\Delta} u_\pm\|_{H^1(\R^d)}=0. 
\]

The equation \eqref{nls} admits a global but nonscattering solution
\[
u(t,x) = e^{it}Q(x), 
\]
where $Q$ is the ground state, namely, the unique positive decaying solution to the elliptic equation
\begin{equation}\label{qeq}
-\Delta Q + Q - |Q|^{\frac{4}{d-1}}Q=0. 
\end{equation}

Duyckaerts, Holmer, and Roudenko \cite{DHR} proved the following scattering result below the ground state threshold in dimension $d=3$.  An analogous result was proven in higher dimensions and for other intercritical nonlinearities in \cite{CFX, Guevara}.

\begin{theorem}\label{T} Suppose $u_0\in H^1(\R^d)$ satisfies
\begin{equation}\label{sub}
M(u_0)E(u_0)< M(Q)E(Q)\qtq{and} \|u_0\|_{L^2}\|u_0\|_{\dot H^1} < \|Q\|_{L^2} \|Q\|_{\dot H^1}.
\end{equation}
Then the solution to \eqref{nls} is global and scatters. 
\end{theorem}

The proof of Theorem~\ref{T} was based on the concentration compactness approach to induction on energy.  We present a simplified proof of Theorem~\ref{T} that avoids concentration compactness. In particular, we establish an interaction Morawetz inequality for solutions to \eqref{nls} obeying \eqref{sub} (see Theorem~\ref{T:IM}).  Combining this with a scattering criterion established in Theorem~\ref{T:SC} then suffices to establish Theorem~\ref{T}.

Theorem~\ref{T} was originally proven in dimension $d=3$ in the radial setting \cite{HR} by Holmer and Roudenko, also through the use of concentration compactness.  In our previous work \cite{DM} we presented a new proof in the radial setting that avoided concentration compactness.  The key was to exploit the radial Sobolev embedding in order to establish a virial/Morawetz hybrid, which (together with a scattering criterion due to Tao \cite{Tao}) was sufficient to prove scattering. 

The extension of Theorem~\ref{T} from the radial to the non-radial setting in \cite{DHR} relied on the Galilean invariance of \eqref{nls}.  In particular, using this symmetry one shows that minimal blowup solutions must have zero momentum, which ultimately ameliorates the lack of compactness due to spatial translation.  In our setting, the key ingredient is an interaction Morawetz inequality for solutions obeying \eqref{sub}, the proof of which also relies on a Galilean transformation.  This estimate is similar to those established in \cite{Dod1, Dod2} for the mass- and energy-critical problems.  In fact, the estimate here is greatly simplified by the fact that the solution belongs to $H^1$.

The scattering in Theorem~\ref{T} is a consequence of the fact that the solution obeys global critical space-time bounds.  In particular, one can verify that our arguments ultimately yield an estimate of the form
\[
\|u\|_{L_{t,x}^{\frac{2(d+2)}{d-1}}(\R\times\R^d)} \lesssim \exp\{r(E(u_0),M(u_0))\}, 
\]
where $r$ is a rational polynomial of $E(u_0)$, $M(u_0)$, $d$, $M(Q)$, and $E(Q)$ (see for example \cite{Tao2}).

The rest of this paper is organized as follows: In Section~\ref{S2}, we set up notation, review some linear theory, and review the variational analysis related to the ground state.  In Section~\ref{S3}, we establish a scattering criterion for \eqref{nls}, Theorem~\ref{T:SC}.  In Section~\ref{S4}, we prove an interaction Morawetz inequality for solutions obeying \eqref{sub}.  Finally, in Section~\ref{S5}, we use the interaction Morawetz inquality to show that solutions obeying \eqref{sub} satisfy the scattering criterion of Theorem~\ref{T:SC}, thereby completing the proof of Theorem~\ref{T}.

\subsection*{Acknowledgements} B. D. was supported by NSF DMS-1500424.

\section{Preliminaries}\label{S2}
We write $A\lesssim B$ to denote $A\leq CB$ for some $C>0$. We use the standard Lebesgue norms
\[
\|f\|_{L_x^r(\R^d)} = \biggl(\int_{\R^d} |f(x)|^r\,dx\biggr)^{\frac1r},\quad \|f\|_{L_t^q L_x^r(I\times\R^d)}=\bigl\| \,\|f(t)\|_{L_x^r(\R^d)}\bigr\|_{L_t^q(I)}, 
\]
with the usual adjustments if $q$ or $r$ is $\infty$. We write $a'\in[1,\infty]$ for the H\"older dual of $a\in[1,\infty]$, i.e. the solution to $\tfrac{1}{a}+\tfrac{1}{a'}=1$. 

\subsection{Local theory and linear estimates} The local theory for \eqref{nls} is standard (see \cite{Caz} for a textbook treatment).  In particular, for any $u_0\in H^1$ there exists a maximal-lifespan solution to \eqref{nls} in $C_t H_x^1$ that conserves mass and energy.  Furthermore, any solution that remains bounded in $H^1$ extends to a global solution. 

The Schr\"odinger group obeys the following dispersive estimates: 
\begin{equation}\label{dispersive}
\|e^{it\Delta}\|_{L_x^{r'}(\R^d)\to L_x^r(\R^d)} \lesssim |t|^{-(\frac{d}{2}-\frac{d}{r})},\quad 2\leq r\leq\infty.
\end{equation}
These estimates imply the standard Strichartz estimates (cf. \cite{GinVel, KeeTao, Str}), which in dimensions $d\geq 3$ take the following form: for any $2\leq q,\tilde q,r,\tilde r\leq\infty$ satisfying
\[
\tfrac{2}{q}+\tfrac{d}{r}=\tfrac{2}{\tilde q}+\tfrac{d}{\tilde r} = \tfrac{d}{2} 
\]
and any $I\subset\R$, we have
\begin{align*}
\| e^{it\Delta}f\|_{L_t^q L_x^r(I\times\R^d)} &\lesssim \|f\|_{L_x^2(\R^d)}, \\
\biggl\| \int_0^t e^{i(t-s)\Delta}F(s)\,ds\biggr\|_{L_t^q L_x^r(I\times\R^d)} &\lesssim \|F\|_{L_t^{\tilde q'}L_x^{\tilde r'}(I\times\R^d)}.
\end{align*}

We also have the standard local smoothing estimates (see e.g. \cite{CS}):
\[
\| \chi_R |\nabla|^{\frac12} e^{it\Delta} \phi \|_{L_{t,x}^2(I\times\R^d)} \lesssim R^{\frac12}\|\phi\|_{L_x^2(\R^d)},
\]
where $I\subset\R$ and $\chi_R$ is a cutoff to a ball of radius $R$.  We will need the following form of local smoothing:
\begin{equation}\label{smoothing2}
\biggl\| \int_0^t e^{i(t-s)\Delta}\chi_R F(s)\,ds\biggr\|_{L_{t,x}^{\frac{2(d+2)}{d-1}}(I\times\R^d)} \lesssim R^{\frac12}\|F\|_{L_{t,x}^2(I\times\R^d)}.
\end{equation}
To deduce this bound, one can first write the dual estimate 
\[
\biggl\| \int_\R |\nabla|^{\frac12}e^{-is\Delta} \chi_R F(s)\,ds\biggr\|_{L_x^2(\R^d)} \lesssim R^{\frac12}\|F\|_{L_{t,x}^2(I\times\R^d)}
\]
and then use Sobolev embedding and Strichartz to deduce 
\[
\biggl\| \int_\R e^{i(t-s)\Delta} \chi_R F(s)\,ds \biggr\|_{L_{t,x}^{\frac{2(d+2)}{d-1}}(I\times\R^d)} \lesssim R^{\frac12} \|F\|_{L_{t,x}^2(I\times\R^d)}.
\]
Then \eqref{smoothing2} follows from the Christ--Kiselev lemma \cite{CK}.

\subsection{Variational analysis}
We briefly review the variational analysis related to the ground state (for more details, see e.g. \cite{Weinstein}). The ground state $Q$ is a solution to \eqref{qeq} and an optimizer of the Gagliardo--Nirenberg inequality:
\begin{equation}\label{E:GN}
\|f\|_{L^{\frac{2(d+1)}{d-1}}}^{\frac{2(d+1)}{d-1}} \leq C_0 \|f\|_{L^2}^{\frac{2}{d-1}} \|\nabla f\|_{L^2}^{\frac{2d}{d-1}},
\end{equation}
where $C_0$ denotes the sharp constant.  Taking the inner product of \eqref{qeq} with $Q$ and $x\cdot\nabla Q$ and integrating by parts yields the following Pohozaev identites:
 \begin{equation}\label{poho}
\begin{aligned}
\|\nabla Q\|_{L^2}^2 + \|Q\|_{L^2}^2 &= \|Q\|_{L^{\frac{2(d+1)}{d-1}}}^{\frac{2(d+1)}{d-1}}, \\
(d-2)\|\nabla Q\|_{L^2}^2 + d\|Q\|_{L^2}^2& = \tfrac{d(d-1)}{d+1}\|Q\|_{L^{\frac{2(d+1)}{d-1}}}^{\frac{2(d+1)}{d-1}}. 
\end{aligned}
\end{equation}
These imply
\[
\|\nabla Q\|_{L^2}^2 = d\|Q\|_{L^2}^2\qtq{and}\|Q\|_{L^{\frac{2(d+1)}{d-1}}}^{\frac{2(d+1)}{d-1}} = (d+1)\|Q||_{L^2}^2,
\]
which allow us to rewrite the sharp constant in \eqref{E:GN} as
\[
C_0 = \tfrac{d+1}{d}\bigl[\|Q\|_{L^2}\|\nabla Q\|_{L^2}]^{-\frac{2}{d-1}}.
\]

In fact, we will need the following slight refinement of \eqref{E:GN}:
\begin{lemma}\label{L:GN} For any $f\in H^1(\R^d)$ and $\xi\in\R^d$, 
\[
\|f\|_{L^{\frac{2(d+1)}{d-1}}}^{\frac{2(d+1)}{d-1}} \leq \frac{d+1}{d}\biggl[ \frac{\|f\|_{L^2}\|\nabla f\|_{L^2}}{\|Q\|_{L^2}\|\nabla Q\|_{L^2}}\biggr]^{\frac{2}{d-1}} \|\nabla[e^{ix\xi}f]\|_{L^2}^2. 
\]
\end{lemma}

\begin{proof} As the $L^{\frac{2(d+1)}{d-1}}$-norm and $L^2$ norm are invariant under $f\mapsto e^{ix\xi}f$,  we deduce
\begin{align*}
\|f\|_{L^{\frac{2(d+1)}{d-1}}}^{\frac{2(d+1)}{d-1}}& \leq \tfrac{d+1}{d} \inf_{\xi\in\R^d}\bigl(\bigl[\tfrac{\|f\|_{L^2}\|\nabla[e^{ix\xi} f]\|_{L^2}}{\|Q\|_{L^2}\|\nabla Q\|_{L^2}}\bigr]^{\frac{2}{d-1}} \|\nabla[e^{ix\xi}f]\|_{L^2}^2\bigr) \\
& \leq \tfrac{d+1}{d}\inf_{\xi\in\R^d}\bigl[\tfrac{\|f\|_{L^2}\|\nabla[e^{ix\xi} f]\|_{L^2}}{\|Q\|_{L^2}\|\nabla Q\|_{L^2}}\bigr]^{\frac{2}{d-1}}\cdot\inf_{\xi\in\R^d}\|\nabla[e^{ix\xi}f]\|_{L^2}^2,
\end{align*}
which implies the desired result. 
\end{proof}

We next record a lemma that is very similar to \cite[Lemma~2.3]{DM} (and hence we omit the proof). The key ingredients are the sharp Gagliardo--Nirenberg inequality, Pohozaev identities, and conservation of mass and energy. 

\begin{lemma}[Coercivity]\label{L:C1} Suppose $u_0\in H^1$ and that for some $\delta\in(0,1)$ we have
\[
M(u_0)E(u_0) < (1-\delta)M(Q)E(Q)\qtq{and} \|u_0\|_{L^2}\|u_0\|_{\dot H^1} \leq \|Q\|_{L^2} \|Q\|_{\dot H^1}.
\]
Let $u:I\times\R^d\to\C$ be the maximal-lifespan solution to \eqref{nls} with $u(0)=u_0$.  Then there exists $\delta'=\delta'(\delta)>0$ so that
\[
\|u(t)\|_{L^2} \|u(t)\|_{\dot H^1} <(1-\delta')\|Q\|_{L^2} \|Q\|_{\dot H^1}\qtq{for}t\in I.
\]
In particular, $I=\R$ and $u$ remains uniformly bounded in $H^1$.  In fact,
\begin{equation}\label{h1bd}
\sup_{t\in\R} \|u(t)\|_{\dot H^1}^2 \lesssim E(u_0). 
\end{equation}
\end{lemma}

\section{A scattering criterion}\label{S3}

In this section we prove a scattering criterion for \eqref{nls}.  Roughly speaking, it states that if in any sufficiently large window of time we can find a large interval on which the scattering norm is small, then the solution must scatter.  The argument is perturbative and relies only on dispersive/Strichartz estimates. 

\begin{theorem}[Scattering criterion]\label{T:SC} Let $u:\R\times\R^d\to\C$ be a solution to \eqref{nls} satisfying
\[
M(u_0)=E(u_0)=E_0\qtq{and} \sup_{t\in\R} \|u(t)\|_{H^1}^2 \lesssim E_0. 
\]

Suppose that 
\begin{equation}\label{sc}
\forall a\in\R\quad\exists t_0\in(a,a+T_0)\qtq{such that}
\| u\|_{L_{t,x}^{\frac{2(d+2)}{d-1}}([t_0-T_0^{\frac13},t_0]\times\R^d)}\lesssim \eps,
\end{equation}
where $\eps=\eps(E_0)$ is sufficiently small and $T_0=T_0(\eps,E_0)$ is sufficiently large.  Then $u$ scatters forward in time.
\end{theorem}

We begin with the following lemma.

\begin{lemma}\label{L:SC} Let $u$ be as in Theorem~\ref{T:SC} and suppose that \eqref{sc} holds.  Then
\begin{equation}\label{sc2}
\begin{aligned}
\forall a\in \R\quad &\exists t_0\in(a,a+T_0)\qtq{such that} \\
&\biggl\| \int_0^{t_0} e^{i(t-s)\Delta}(|u|^{\frac{4}{d-1}}u)(s)\,ds\biggr\|_{L_t^{\frac{2(d+1)}{d}} L_x^{\frac{2d(d+1)}{d^2-2d-1}}([t_0,\infty)\times\R^d)}\lesssim \eps^{\frac{4}{d-1}}. 
\end{aligned}
\end{equation}
\end{lemma}

\begin{proof}[Proof of Lemma~\ref{L:SC}] Let $a\in\R$ and choose $t_0\in(a,a+T_0)$ as in \eqref{sc}.  By Strichartz and standard bootstrap arguments, we can use the estimate in \eqref{sc} to deduce
\[
\| |\nabla|^{\frac12} u\|_{L_t^2 L_x^{\frac{2d}{d-2}}([t_0-T_0^{\frac13},t_0]\times\R^d)} \lesssim 1.
\]

Thus, by Sobolev embedding and Strichartz,
\begin{align*}
\biggl\|& \int_{t_0-T_0^{\frac13}}^{t_0} e^{i(t-s)\Delta}(|u|^{\frac{4}{d-1}}u)(s)\,ds \biggr\|_{L_t^{\frac{2(d+1)}{d}} L_x^{\frac{2d(d+1)}{d^2-2d-1}}([t_0,\infty)\times\R^d)} \\
&\lesssim \|u\|_{L_{t,x}^{\frac{2(d+2)}{d-1}}([t_0-T_0^{\frac13},t_0]\times\R^d)}^{\frac{4}{d-1}} \| |\nabla|^{\frac12} u\|_{L_t^2 L_x^{\frac{2d}{d-2}}([t_0-T_0^{\frac13},t_0]\times\R^d)} \lesssim \eps^{\frac{4}{d-1}}. 
\end{align*}

We treat the remaining part of the integral by interpolation.  First, by the dispersive estimate \eqref{dispersive} and $2\leq \frac{2(d+1)}{d-1}\leq\frac{2d}{d-2}$, 
\begin{align*}
\| e^{i(t-s)\Delta}(|u|^{\frac{4}{d-1}}u)(s)\|_{L^{\frac{2(d+1)}{d-3}}}& \lesssim |t-s|^{-\frac{2d}{d+1}} \|u(s)\|_{L^{\frac{2(d+1)}{d-1}}}^{\frac{4}{d-1}} \|u(s)\|_{L^2}\\
& \lesssim  |t-s|^{-\frac{2d}{d+1}}E_0^{\frac{d+3}{2(d-1)}}. 
\end{align*}
Thus
\[
\biggl\| \int_0^{t_0-T_0^{\frac13}} e^{i(t-s)\Delta}(|u|^{\frac{4}{d-1}} u)(s)\,ds\biggr\|_{L_t^{\frac{2(d+1)}{d-1}} L_x^{\frac{2(d+1)}{d-3}}([t_0,\infty)\times\R^d)} \lesssim E_0^{\frac{d+3}{2(d-1)}} T_0^{-\frac{d-1}{6(d+1)}}. 
\]
Next, noting that
\[
i\int_0^{t_0-T_0^{\frac13}}e^{i(t-s)\Delta}(|u|^{\frac{4}{d-1}}u)(s)\,ds = e^{i(t-t_0+T_0^{\frac13})\Delta}u(t_0-T_0^{\frac13}) - u(0), 
\]
we use Strichartz to estimate
\begin{align*}
\biggl\|& \int_0^{t_0-T_0^{\frac13}} e^{i(t-s)\Delta}(|u|^{\frac{4}{d-1}} u)(s)\,ds\biggr\|_{L_t^2 L_x^{\frac{2d}{d-2}}([t_0,\infty)\times\R^d)} \\
& \lesssim  \|u(t_0-T_0^{\frac13})\|_{L^2(\R^d)} + \|u(0)\|_{L^2(\R^d)} \lesssim E_0^{\frac12}. 
\end{align*}
Thus, by interpolation, we have
\begin{align*}
\biggl\| \int_0^{t_0}& e^{i(t-s)\Delta} (|u|^{\frac{4}{d-1}}u)(s)\,ds\biggr\|_{L_t^{\frac{2(d+1)}{d}} L_x^{\frac{2d(d+1)}{d^2-2d-1}}([t_0,\infty)\times\R^d)}\\
&\lesssim \eps^{\frac{4}{d-1}}+T_0^{-\frac{d-1}{12(d+1)}}E_0^{\frac{d+1}{2(d-1)}},
\end{align*}
and hence \eqref{sc2} holds for $T_0=T_0(E_0,\eps)$ sufficiently large. 
\end{proof}

We now turn to the proof of Theorem~\ref{T:SC}. 

\begin{proof}[Proof of Theorem~\ref{T:SC}] Let $u$ be as in the statement of Theorem~\ref{T:SC}.  We suppose that \eqref{sc} holds, and hence (by Lemma~\ref{L:SC}) \eqref{sc2} holds as well.

We begin by splitting $\R$ into $J=J(\eps,E_0)$ intervals $I_j$ such that
\begin{equation}\label{ij}
\| e^{it\Delta} u_0\|_{L_t^{\frac{2(d+1)}{d}} L_x^{\frac{2d(d+1)}{d^2-2d-1}}(I_j\times\R^d)} < \eps. 
\end{equation}
Our goal is to prove that for $T_0$ large enough, we have
\begin{equation}\label{goal}
\|u\|_{L_t^{\frac{2(d+1)}{d}} L_x^{\frac{2d(d+1)}{d^2-2d-1}}(I_j\times\R^d)}^{\frac{2(d+1)}{d}} \lesssim_{E_0} T_0
\end{equation} 
for each $j$.  Once we have established \eqref{goal}, summation over $I_j\subset\R$ yields the critical global space-time bound
\[
\|u\|_{L_t^{\frac{2(d+1)}{d}} L_x^{\frac{2d(d+1)}{d^2-2d-1}}(\R\times\R^d)}^{\frac{2(d+1)}{d}} \lesssim_{E_0} T_0,
\]
and a standard argument then yields scattering. 

We turn to \eqref{goal}.  In light of the general bound
\begin{equation}\label{gen-bd}
\|u\|_{L_t^{\frac{2(d+1)}{d}} L_x^{\frac{2d(d+1)}{d^2-2d-1}}(I\times\R^d)}^{\frac{2(d+1)}{d}}\lesssim_{E_0} \langle I\rangle
\end{equation}
(see e.g. \cite[Section~3]{Tao}), it suffices to consider $j$ such that $|I_j|> 2T_0$. 

Therefore we fix some $I_j=(a_j,b_j)$ with $|I_j|> 2T_0$.  We then choose $t_0\in(a_j,a_j+T_0)$ such that the estimate in \eqref{sc2} holds.  Writing
\[
e^{i(t-t_0)\Delta}u(t_0) = e^{it\Delta}u_0 + i\int_0^{t_0} e^{i(t-s)\Delta}(|u|^{\frac{4}{d-1}}u)(s)\,ds
\]
we can use \eqref{sc2} and \eqref{ij} to deduce
\[
\| e^{i(t-t_0)\Delta}u(t_0)\|_{L_t^{\frac{2(d+1)}{d}} L_x^{\frac{2d(d+1)}{d^2-2d-1}}((t_0,b_j)\times\R^d)} \lesssim \eps^{c}
\]
for some $0<c\leq 1$. Choosing $\eps$ small enough, a standard continuity argument then yields
\[
\|u\|_{L_t^{\frac{2(d+1)}{d}} L_x^{\frac{2d(d+1)}{d^2-2d-1}}((t_0,b_j)\times\R^d)} \lesssim \eps^c,
\]
and hence (using $t_0-a_j<T_0$ and \eqref{gen-bd} to handle the contribution of $(a_j,t_0)$) we deduce \eqref{goal} and complete the proof of Theorem~\ref{T:SC}. \end{proof} 

\section{Interaction Morawetz inequality}\label{S4}

In this section we prove an interaction Morawetz inequality for solutions to \eqref{nls} obeying \eqref{sub}.  The interaction Morawetz inequality was originally introduced in \cite{CKSTT} in the defocusing setting.  The estimate we prove is similar to those established in \cite{Dod1} and \cite{Dod2} in the focusing mass- and energy-critical settings.  In order to exhibit coercivity, we rely on the sharp Gagliardo--Nirenberg inequality and Galilean invariance. 

Let $\eps>0$ be a small constant and let $\chi$ be a radial decreasing function satisfying
\begin{equation}\label{chi}
\chi(x) =\begin{cases} 1 & |x|\leq 1-\eps, \\ 0 & |x|> 1\end{cases}. 
\end{equation}

Let $R\gg 1$ and define the radial function 
\[
\phi(x) = \tfrac{1}{\omega_d R^d} \int_{\R^d} \chi^2(\tfrac{x-s}{R})\chi^2(\tfrac{s}{R})\,ds, 
\]
where $\omega_d$ is the volume of the unit ball in $\R^d$.  Finally, introduce
\[
\psi(x)=\tfrac{1}{|x|}\int_0^{|x|} \phi(r)\,dr. 
\]
Note that 
\begin{equation}\label{1211}
|\psi(x)| \lesssim \min\{1,\tfrac{R}{|x|}\}\qtq{and} \partial_k \psi(x) = \frac{x_k}{|x|^2}[\phi(x)-\psi(x)]. 
\end{equation}

For a global solution $u$ to \eqref{nls}, we introduce the interaction Morawetz quantity
\[
M_R(t) = \iint_{\R^d\times\R^d} |u(t,y)|^2\psi(x-y)(x-y)\cdot 2\Im[\bar u\nabla u](t,x)\,dx\,dy. 
\]
Using \eqref{1211}, we see that
\begin{equation}\label{mrbd}
\sup_{t\in\R} |M_R(t)| \lesssim RE_0^2.
\end{equation}

Using this quantity, we will establish the following estimate.
\begin{theorem}[Interaction Morawetz estimate]\label{T:IM} Let $u_0\in H^1$ satisfy \eqref{sub}, and suppose further that
\[
M(u_0)=E(u_0)=E_0.
\]
Let $u:\R\times\R^d\to\C$ be the corresponding global solution to \eqref{nls} (cf. Lemma~\ref{L:C1}). 

Let $\eps>0$, $a\in\R$, $T_0\geq 1$, and $J\geq \eps^{-1}$.  There exists $\delta>0$ such that for $R_0=R_0(\delta,M(u),Q)$ sufficiently large,
\[
\tfrac{\delta}{J T_0}\!\int_a^{a+T_0}\int_{R_0}^{R_0e^J}\!\tfrac{1}{R^d} \iiint |\chi(\tfrac{y-s}{R})u(y)|^2 |\chi(\tfrac{x-s}{R})\nabla[e^{ix\xi} u(x)]|^2\,dx\,dy\,ds\tfrac{dR}{R}\,dt\lesssim \nu E_0^2,
\]
where $\chi$ is as in \eqref{chi}, 
\begin{equation}\label{nu}
\nu= \tfrac{R_0 e^J}{T_0J} + \eps,
\end{equation}
and
\[
\xi=\xi(t,s,R)= -\frac{\int \chi^2(\frac{x-s}{R})\Im[\bar u\nabla u](t,x) \,dx}{\int \chi^2(\frac{x-s}{R})|u(t,x)|^2\,dx}
\]
(unless the denominator is zero, in which case $\xi(t,s,R)=0$). 
\end{theorem}

\begin{proof}

In the following, repeated indices are summed and explicit dependence on $t$ is suppressed. We write $u_j=\partial_j u$, and so on.  

We rely on the identities
\begin{align*}
\partial_t 2\Im \bar u u_k &= \partial_k \tfrac{4}{d+1}|u|^{\frac{2(d+1)}{d-1}} + \partial_{jjk}|u|^2 - 4\Re\partial_j(\bar u_j u_k), \\
\partial_t |u|^2 &= -2\partial_k \Im(\bar u u_k),
\end{align*}
which follow from the equation \eqref{nls}. Then 
\begin{align}
\tfrac{dM_R}{dt} & =\tfrac{4}{d+1} \iint |u(y)|^2 \psi(x-y)(x-y)_k \partial_{k} |u(x)|^{\frac{2(d+1)}{d-1}} \,dx\,dy \label{mor1} \\ 
&\quad +\iint |u(y)|^2 \psi(x-y)(x-y)_k \partial_{jjk}|u(x)|^2 \,dx \,dy \label{mor2} \\
&\quad -4\iint \partial_{j}\Im(\bar u u_j)(y) \psi(x-y)(x-y)_k \Im(\bar u u_k)(x)\,dx\,dy \label{mor3} \\
&\quad -4\iint |u(y)|^2 \psi(x-y)(x-y)_k \Re\partial_{j}(\bar u_j u_k)(x)\,dx\,dy. \label{mor4}
\end{align}

We will first integrate by parts in \eqref{mor1}. Using \eqref{1211}, we have
\[
\partial_k[\psi(x) x_k] = d\phi(x) +(d-1)(\psi-\phi)(x). 
\]
We further introduce
\begin{equation}\label{phi1}
\phi_1(x,y) = \tfrac{1}{\omega_d R^d} \int \chi^2(\tfrac{y-s}{R}) \chi^{\frac{2(d+1)}{d-1}}(\tfrac{x-s}{R}).
\end{equation}
We can then write 
\begin{align}
\eqref{mor1} &= -\tfrac{4d}{(d+1)\omega_d R^d}\iiint \chi^2(\tfrac{y-s}{R})\chi^{\frac{2(d+1)}{d-1}}(\tfrac{x-s}{R})|u(y)|^2 |u(x)|^{\frac{2(d+1)}{d-1}}\,dx\,dy\,ds \label{mor10} \\
&\quad -\tfrac{4(d-1)}{d+1} \iint |u(y)|^2 |u(x)|^{\frac{2(d+1)}{d-1}}[(\psi-\phi)(x-y)]\,dx\,dy \label{mor11}\\
&\quad  -\tfrac{4d}{d+1}\iint |u(y)|^2 |u(x)|^{\frac{2(d+1)}{d-1}}[\phi(x-y)-\phi_1(x,y)]\,dx\,dy.\label{mor12}
\end{align}
We will later use \eqref{mor10} and the sharp Gagliardo--Nirenberg inequality (along with \eqref{75} below) to exhibit some coercivity, while \eqref{mor11} and \eqref{mor12} will be treated as error terms.

We next consider \eqref{mor2}.  We integrate by parts twice, which yields
\begin{equation}\label{mor13}
\eqref{mor2} = \iint |u(y)|^2 \nabla|u(x)|^2 \cdot \nabla[(d-1)\psi(x-y)+\phi(x-y)]\,dx\,dy.
\end{equation}
This term will be treated as an error term below. 

We turn our attention to \eqref{mor3} and \eqref{mor4}.  Let us denote
\[
P_{jk}(x-y)=\delta_{jk}-\tfrac{(x-y)_j(x-y)_k}{|x-y|^2}.
\]
Then integrating by parts,
\begin{align}
\eqref{mor3} & = 4\iint \Im(\bar u u_j)(y) \partial_j[\psi(x-y)(x-y)_k]\Im(\bar u u_k)(x)\,dx\,dy \nonumber \\
& = -4\iint \Im(\bar u u_j)(y)\Im(\bar u u_k)(x)\delta_{jk}\phi(x-y)\,dx\,dy\nonumber \\
& \quad -4\iint \Im(\bar u u_j)(y)\Im(\bar u u_k)(x)P_{jk}(x-y)[(\psi-\phi)(x-y)]\,dx\,dy\nonumber \\
& = -\tfrac{4}{\omega_d R^d} \iiint \chi^2(\tfrac{x-s}{R})\chi^2(\tfrac{y-s}{R})\Im[\bar u\nabla u](x)\cdot\Im[\bar u\nabla u](y)\,dx\,dy\,ds\label{mor5} \\
& \quad -4\iint \Im(\bar u u_j)(y)\Im(\bar u u_k)(x)P_{jk}(x-y)[(\psi-\phi)(x-y)]\,dx\,dy.\label{mor6}
\end{align}

Integrating by parts in \eqref{mor4}, we compute similarly
\begin{align}
\eqref{mor4} & = \tfrac{4}{\omega_d R^d}\iiint \chi^2(\tfrac{x-s}{R})\chi^2(\tfrac{y-s}{R})|u(y)|^2 |\nabla u(x)|^2 \,dx\,dy\,ds \label{mor7} \\
&\quad + 4\iint |u(y)|^2 \Re(\bar u_j u_k)(x) P_{jk}(x-y)[(\psi-\phi)(x-y)]\,dx\,dy. \label{mor8}
\end{align}

We treat the contribution of $\eqref{mor8}+\eqref{mor6}$ and $\eqref{mor7}+\eqref{mor5}$ separately. 

First, let $\slashed{\nabla}_y$ denote the angular derivative centered at $y$.  Then
\begin{align*}
\eqref{mor8}+\eqref{mor6} & = 4\iint |u(y)|^2 | |\slashed{\nabla}_y u(x)|^2[(\psi-\phi)(x-y)]\,dx\,dy \\
& \quad -4\iint \Im[\bar u\slashed{\nabla}_x u](y)\cdot\Im [\bar u \slashed{\nabla}_y u](x)[(\psi-\phi)(x-y)]\,dx\,dy,
\end{align*}
and hence by Cauchy--Schwarz and the fact that $\psi-\phi$ is a nonnegative radial function, we deduce
\begin{equation}\label{860}
\eqref{mor8}+\eqref{mor6}\geq 0. 
\end{equation}

We turn to $\eqref{mor7}+\eqref{mor5}$. For fixed $s\in\R^d$, consider the quantity  defined by
\[
\iint \chi^2(\tfrac{x-s}{R})\chi^2(\tfrac{y-s}{R})\bigl\{ |u(y)|^2|\nabla u(x)|^2\! - \Im[\bar u\nabla u](x)\cdot\Im[\bar u\nabla u](y)\bigr\}\,dx\,dy. 
\]
We claim this quantity is Galilean invariant, that is, invariant under the transformation 
\[
u(t,x)\mapsto u^{\xi}(t,x):=e^{ix\xi}u(t,x)
\]
for any $\xi=\xi(t,s,R)$.  Indeed, one has
\begin{align*}
|&u^\xi(y)|^2|\nabla u^\xi(x)|^2 - \Im[\bar u^\xi\nabla u^\xi](x)\cdot\Im[\bar u^\xi\nabla u^\xi](y) \\
& =|u(y)|^2 |\nabla u(x)|^2 - \Im[\bar u\nabla u](x)\cdot\Im[\bar u\nabla u](y) \\
&\quad  +\xi \cdot |u(y)|^2 \Im [\bar u\nabla u](x) - \xi\cdot|u(x)|^2\Im[\bar u\nabla u](y) \end{align*}
and hence the claim follows by symmetry of $\chi^2$ and a change of variables.

We now define $\xi=\xi(t,s,R)$ so that
\[
\int \chi^2(\tfrac{x-s}{R})\Im[ \bar u^\xi\nabla u^\xi](x)\,dx = 0. 
\]
In particular, we can achieve this by choosing
\[
\xi(t,s,R)= -\frac{\int \chi^2(\frac{x-s}{R})\Im[\bar u\nabla u](x) \,dx}{\int \chi^2(\frac{x-s}{R})|u(x)|^2\,dx} 
\]
provided the denominator is nonzero (otherwise $\xi\equiv 0$ suffices).

For this choice of $\xi$, we have
\begin{equation}\label{75}
\eqref{mor7}+\eqref{mor5} = \tfrac{4}{\omega_dR^d} \iiint \chi^2(\tfrac{x-s}{R})\chi^2(\tfrac{y-s}{R}) |u(y)|^2 |\nabla u^\xi(x)|^2\,dx\,dy\,ds.
\end{equation}
We will combine this term with \eqref{mor10} and use the sharp Gagliardo--Nirenberg inequality to exhibit coercivity below. 

We now collect \eqref{mor10}, \eqref{mor11}, \eqref{mor12}, \eqref{mor13}, \eqref{860}, and \eqref{75} to deduce 
\begin{align}
&\tfrac{c}{R^d} \! \iiint |\chi(\tfrac{y-s}{R}) u(y)|^2 \bigl\{ |\chi(\tfrac{x-s}{R})\nabla u^\xi(x)|^2\!-\!\tfrac{d}{d+1}|\chi(\tfrac{x-s}{R}) u(x)|^{\frac{2(d+1)}{d-1}}\bigr\}\,dx\,dy\,ds \label{mor-lhs}\\ & \leq \frac{dM_R}{dt} \label{morq} \\
& \quad + \iint |u(y)|^2 |u(x)|^{\frac{2(d+1)}{d-1}} \bigl|[\psi-\phi](x-y)+ [\phi(x-y)-\phi_1(x,y)]\bigr|\,dx\,dy \label{mor-error1}\\
& \quad + \iint |u(y)|^2 |u(x)| |\nabla u(x)| \bigl| \nabla[\psi(x-y)+\phi(x-y)]\bigr| \,dx\,dy  \label{mor-error2}
\end{align}
for some $c>0$.  

We will average this inequality over $t\in [a,a+T_0]$ and logarithmically over $R\in[R_0,R_0 e^J]$.

We start with \eqref{morq}.  Recalling \eqref{mrbd}, we have by the fundamental theorem of calculus
\begin{equation}\label{averaged-1}
\biggl| \tfrac{1}{ JT_0} \int_a^{a+T_0} \int_{R_0}^{R_0e^J} \tfrac{dM_R}{dt} \tfrac{dR}{R} dt\biggr| \lesssim \tfrac{1}{T_0} \tfrac{R_0 e^J}{J} E_0^2. 
\end{equation}

We turn to \eqref{mor-error1}.  Recalling \eqref{phi1} and the definition of $\chi$, we have
\[
\bigl| \phi(x-y)-\phi_1(x,y)\bigr| \lesssim \eps.
\]
Similarly, by construction,
\[
|\psi(x)-\phi(x)| \lesssim \min\{\tfrac{|x|}{R},\tfrac{R}{|x|}\}. 
\]
Noting that
\[
\tfrac{1}{J}\int_{R_0}^{e^JR_0} \min\bigl\{\tfrac{|x-y|}{R},\tfrac{R}{|x-y|}\bigr\}\tfrac{dR}{R} \lesssim \tfrac{1}{J}, 
\]
we deduce
\begin{equation}\label{averaged-2}
\tfrac{1}{T_0}\int_I \tfrac{1}{J}\int_{R_0}^{e^JR_0} \eqref{mor-error1}\tfrac{dR}{R}\,dt \lesssim (\eps+\tfrac{1}{J})E_0^2. 
\end{equation}

We turn to \eqref{mor-error2}.  For this term we note that
\[
|\nabla\phi|\lesssim \tfrac{1}{R}\qtq{and}|\nabla\phi| = \bigl|\tfrac{x}{|x|^2}(\phi-\psi)\bigr| \lesssim \min\{\tfrac{1}{R},\tfrac{R}{|x|^2}\}.
\]
Thus, noting that
\[
\tfrac{1}{J}\int_{R_0}^{e^JR_0}\tfrac{dR}{R^2} \lesssim \tfrac{1}{JR_0}E_0^2,
\]
we deduce
\begin{equation}\label{averaged-3}
\tfrac{1}{T_0}\int_I \tfrac{1}{J}\int_{R_0}^{e^JR_0}\eqref{mor-error2}\tfrac{dR}{R}\,dt \lesssim \tfrac{1}{JR_0}E_0^2. 
\end{equation}

Finally, we consider \eqref{mor-lhs}.  We wish to establish a {lower} bound for this term by using the sharp Gagliardo--Nirenberg inequality.

To this end, first note that by \eqref{sub} and Lemma~\ref{L:C1}, there exists $\delta>0$ so that 
\begin{equation}\label{below1}
\sup_{t\in\R}\bigl\{\|u(t)\|_{L^2} \|\nabla u(t)\|_{L^2}\bigr\} < (1-3\delta)^{\frac{d-1}{2}} \|Q\|_{L^2}\|\nabla Q\|_{L^2}. 
\end{equation}
We claim that for $R=R(\delta,M(u),Q)$ large enough, we in fact have
\begin{equation}\label{below2}
\sup_{t\in\R}\bigl\{\|\chi(\tfrac{\cdot-s}{R})u\|_{L^2}\|\nabla[\chi(\tfrac{\cdot-s}{R})u]\|_{L^2}\bigr\}<(1-2\delta)^{\frac{d-1}{2}}\|Q\|_{L^2}\|\nabla Q\|_{L^2}. 
\end{equation}
To see this, first note that multiplication by $\chi$ only decreases the $L^2$-norm, so that it suffices to consider the $\dot H^1$-norm.  For this, we use the identity
\begin{equation}\label{IBP}
\int |\nabla(\chi u)|^2 \,dx = \int \chi^2|\nabla u|^2 - \chi\Delta \chi |u|^2\,dx,
\end{equation}
whence
\[
\|\nabla[\chi(\tfrac{\cdot-s}{R})u]\|_{L^2}^2 \leq \|\nabla u\|_{L^2}^2 + \mathcal{O}(\tfrac{1}{R^2}M(u)).
\]
In particular, choosing $R=R(\delta,M(u),Q)$ large enough, \eqref{below2} follows from \eqref{below1}.  

In light of \eqref{below1}, we may apply the sharp Gagliardo--Nirenberg inequality in the form of Lemma~\ref{L:GN} (and then use \eqref{IBP} once more, possibly choosing $R=R(M(u),\delta)$ even larger) to deduce 
\[
\tfrac{d}{d+1}\| \chi(\tfrac{\cdot-s}{R}) u\|_{L^{\frac{2(d+1)}{d-1}}}^{\frac{2(d+1)}{d-1}} \leq (1-2\delta)\| \nabla[\chi(\tfrac{\cdot-s}{R}) u^\xi]\|_{L^2}^2\leq (1-\delta)\|\chi(\tfrac{\cdot-s}{R})\nabla u^\xi\|_{L^2}^2
\]
uniformly for $t\in\R$. Consequently,
\begin{equation}\label{averaged-4}
\begin{aligned}
\tfrac{1}{T_0}&\int_I \tfrac{1}{J}\int_{R_0}^{R_0e^J}\eqref{mor-lhs}\tfrac{dR}{R}\,dt \\
& \geq\tfrac{\delta}{T_0}\int_I\tfrac{1}{J}\int_{R_0}^{R_0e^J}\tfrac{1}{R^d} \iiint |\chi(\tfrac{y-s}{R})u(y)|^2 |\chi(\tfrac{x-s}{R})\nabla u^\xi(x)|^2\,dx\,dy\,ds\tfrac{dR}{R}\,dt.
\end{aligned}
\end{equation}

Collecting \eqref{averaged-1}, \eqref{averaged-2}, \eqref{averaged-3}, and \eqref{averaged-4}, we find
\begin{equation}\label{final-mor}
\begin{aligned}
\tfrac{\delta}{T_0}&\int_I\tfrac{1}{J}\int_{R_0}^{R_0e^J}\tfrac{1}{R^d} \iiint |\chi(\tfrac{y-s}{R})u(y)|^2 |\chi(\tfrac{x-s}{R})\nabla u^\xi(x)|^2\,dx\,dy\,ds\tfrac{dR}{R}\,dt\\
& \lesssim E_0^2( \tfrac{R_0 e^J}{T_0J} + \eps+\tfrac{1}{J} + \tfrac{1}{R_0J}),
\end{aligned}
\end{equation}
which completes the proof of Theorem~\ref{T:IM}. \end{proof}

\section{Proof of the main result}\label{S5}

In this section, we use the interaction Morawetz inequality and the scattering criterion in Theorem~\ref{T:SC} to prove Theorem~\ref{T}.  

\begin{proof}[Proof of Theorem~\ref{T}] Let $u_0$ be as Theorem~\ref{T} and let $u$ be the corresponding solution to \eqref{nls}.  By Lemma~\ref{L:C1}, $u$ is global and obeys \eqref{h1bd}. It remains to establish scattering.  In particular, it suffices to consider scattering forward in time, for which it is enough to verify the scattering criterion \eqref{sc} appearing in Theorem~\ref{T:SC}.

Using the rescaling \eqref{scale}, we may assume that
\[
M(u_0)=E(u_0)=E_0.
\] 
In order to establish \eqref{sc}, we first fix $a\in\R$ and let $0<\eps\ll 1$ and $T_0\gg1$ (to be determined more precisely below).

By Theorem~\ref{T:IM}, there exists $t_1\in(a, a+\tfrac12T_0)$ and $R\in[R_0,e^JR_0]$ so that
\[
\tfrac{\delta}{R^d}\iiint |\chi(\tfrac{y-s}{R})u(t_1,y)|^2 |\chi(\tfrac{x-s}{R})\nabla u^{\xi_1}(t_1,x)|^2\,dx\,dy\,ds \lesssim \nu E_0^2,
\]
where $\xi_1=\xi(t_1,s,R)$ and $\nu$ is as in \eqref{nu}. We make the change of variables $s=\tfrac{R}{4}(z+\theta)$ where $z\in\Z^d$ and $\theta\in[0,1]^d$.  We deduce that there exists $\theta_1\in[0,1]^d$ so that
\begin{equation}\label{im-conclusion} 
\delta\sum_{z\in\Z^d} \iint |\chi(\tfrac{y-s_1}{R})u(t_1,y)|^2 |\chi(\tfrac{x-s_1}{R})|\nabla u^{\xi_1}(t_1,x)|^2\,dx\,dy\lesssim \nu E_0^2,
\end{equation}
where now
\[
s_1=s_1(z)=\tfrac{R}{4}z+\theta_1 \qtq{and}\xi_1=\xi(t_1,s_1(z),R). 
\]
Without loss of generality, we may assume $\theta_1=0$, so that 
\[
s_1=\tfrac{R}{4}z\qtq{and}\xi_1=\xi(t_1,\tfrac{R}{4}z,R). 
\]

Our goal will be to estimate
\[
\| u_L(t) \|_{L_{t,x}^{\frac{2(d+2)}{d-1}}([t_1,t_1+T_0^{\frac13}]\times\R^d)},\qtq{where} u_L(t):=e^{i(t-t_1)\Delta}u(t_1).
\]
To this end, for each $z\in\Z^d$, we introduce $v(z):\R\times\R^d\to\C$ defined by
\begin{equation}\label{vz}
v(z;t,x) := \chi(\tfrac{x-\frac{R}{4}z}{R})u_L(t).
\end{equation}
By the support properties of $\chi$, we have
\begin{equation}\label{vz2}
\| u_L \|_{L_{t,x}^{\frac{2(d+2)}{d-1}}([t_1,t_1+T_0^{\frac13}]\times\R^d)}^{\frac{2(d+2)}{d-1}}\lesssim\sum_{z\in\Z^d} \|v(z)\|_{L_{t,x}^{\frac{2(d+2)}{d-1}}([t_1,t_1+T_0^{\frac13}]\times\R^d)}^{\frac{2(d+2)}{d-1}},
\end{equation}
and hence we are faced with estimating each $v(z)$.  To this end, note that each $v(z)$ solves an equation, namely
\[
\begin{cases}
(i\partial_t+\Delta)v(z) =  2\nabla u_L \cdot\nabla[\chi(\tfrac{x-\frac{R}{4}z}{R})] + u_L\Delta[\chi(\tfrac{x-\frac{R}{4}z}{R})], \\
v(z;t_1,x) = \chi(\tfrac{x-\frac{R}{4}z}{R})u(t_1,x).
\end{cases}
\]

To simplify notation, let us translate $t_1$ to $0$, take space-time norms over $[0,T_0^{\frac13}]\times\R^d$, and denote $q=\frac{2(d+2)}{d-1}$. 

We then have the following Duhamel formula for $v(z)$: 
\begin{align}
v(z;t) & = e^{it\Delta}\bigl[\chi(\tfrac{\cdot-\frac{R}{4}z}{R})u(0)\bigr] \label{v-duhamel1}\\
& \quad -i\int_{0}^t e^{i(t-s)\Delta}\bigl\{ 2\nabla u_L(s)\cdot \nabla[\chi(\tfrac{\cdot-\frac{R}{4}z}{R})] + u_L(s)\Delta[\chi(\tfrac{\cdot-\frac{R}{4}z}{R})]\bigr\}\,ds. \label{v-duhamel3}
\end{align}

To estimate \eqref{v-duhamel1}, we will exploit Galilean invariance, specifically through the identity
\[
e^{it\Delta}[e^{ix\xi}f(x)] = e^{-it|\xi|^2 + ix\xi}[e^{it\Delta}f](x-2t\xi)\qtq{for any}\xi\in\R^d. 
\]
Thus by a change of variables, Sobolev embedding, Strichartz, and interpolation, we deduce
\begin{align}
\| e^{it\Delta}[\chi(\tfrac{\cdot-\frac{R}{4}z}{R})u(0)]\|_{L_{t,x}^q}& \lesssim \|e^{it\Delta} |\nabla|^{\frac12}[e^{ix\xi_0}\chi(\tfrac{\cdot-\frac{R}{4}z}{R})u(0)]\|_{L_t^q L_x^{\frac{2dq}{dq-4}}} \nonumber \\
&\lesssim \bigl\| \chi(\tfrac{\cdot-\frac{R}{4}z}{R})u(0)\bigr\|_{L^2}^{\frac12} \bigl\|\nabla\bigl[\chi(\tfrac{x-\frac{R}{4}z}{R})e^{ix\xi_0}u(0)\bigr]\bigr\|_{L^2}^{\frac12}, \nonumber \\
& \lesssim \bigl\| \chi(\tfrac{\cdot-\frac{R}{4}z}{R})u(0)\bigr\|_{L^2}^{\frac12} \bigl\|\chi(\tfrac{x-\frac{R}{4}z}{R})\nabla[e^{ix\xi_0}u(0)]\bigr\|_{L^2}^{\frac12} \label{1218-bd1} \\
& \quad + \bigl\| \chi(\tfrac{\cdot-\frac{R}{4}z}{R})u(0)\bigr\|_{L^2}^{\frac12} \bigl\|\nabla[\chi(\tfrac{x-\frac{R}{4}z}{R})]u(0)\bigr\|_{L^2}^{\frac12}. \label{1218-bd2}
\end{align}
For \eqref{1218-bd1}, we use H\"older and \eqref{im-conclusion} to estimate
\begin{align}
\| \eqref{1218-bd1}\|_{\ell_z^q} & \lesssim \| \chi(\tfrac{\cdot-\frac{R}{4}z}{R})u(0) \|_{\ell_z^q L_x^2}^{\frac12} \|\chi(\tfrac{x-\frac{R}{4}z}{R})\nabla[e^{ix\xi_0}u(0)]\|_{\ell_z^q L_x^2}^{\frac12} \nonumber \\
& \lesssim  \| \chi(\tfrac{\cdot-\frac{R}{4}z}{R})u(0) \|_{\ell_z^2 L_x^2}^{\frac12} \|\chi(\tfrac{x-\frac{R}{4}z}{R})\nabla[e^{ix\xi_0}u(0)]\|_{\ell_z^2 L_x^2}^{\frac12}  \lesssim \delta^{-\frac14}\nu^{\frac14}E_0^{\frac12}. \label{1218-3}
\end{align}
For \eqref{1218-bd2}, we instead have
\begin{align}
\| \eqref{1218-bd2}\|_{\ell_z^q} & \lesssim \tfrac{1}{\eps^{\frac12} R^{\frac12}} \|\chi(\tfrac{\cdot-\frac{R}{4}z}{R})u(0)\|_{\ell_z^2 L_x^2}^{\frac12} \| (\nabla \chi)(\tfrac{\cdot-\frac{R}{4}z}{R})u(0)\|_{\ell_z^2 L_x^2}^{\frac12} \lesssim \frac{E_0^{\frac12}}{\eps^{\frac12} R^{\frac12}}. \label{1218-4}
\end{align}

We estimate \eqref{v-duhamel3} using the local smoothing estimate \eqref{smoothing2}:
\begin{align*}
\| \eqref{v-duhamel3}\|_{\ell_z^q L_{t,x}^q} & \lesssim R^{\frac12}\| \nabla u_L\cdot \nabla[\chi(\tfrac{\cdot-\frac{R}{4}z}{R})]\|_{\ell_z^q L_{t,x}^2} + R^{\frac12}\| u_L\Delta[\chi(\tfrac{\cdot-\frac{R}{4}z}{R})]\|_{\ell_z^q L_{t,x}^2} \\
& \lesssim R^{\frac12}\| \nabla u_L\cdot \nabla[\chi(\tfrac{\cdot-\frac{R}{4}z}{R})]\|_{\ell_z^2 L_{t,x}^2} + R^{\frac12}\| u_L\Delta[\chi(\tfrac{\cdot-\frac{R}{4}z}{R})]\|_{\ell_z^2 L_{t,x}^2}.
\end{align*}
Then by the support properties of $\chi$, we can bound
\[
R\int_0^{T_0^{\frac13}} \sum_{z\in\Z^d} \bigl\|\nabla u_L(t)\cdot\nabla[\chi(\tfrac{\cdot-\frac{R}{4}z}{R})]\bigr\|_{L_x^2}^2\,dt  \lesssim \tfrac{1}{\eps^2 R} \int_0^{T_0^{\frac13}} \|\nabla_L u(t)\|_{L_x^2}^2 \,dt \lesssim \tfrac{T_0^{\frac13}}{\eps^2 R}E_0.
\]
Estimating similarly for the term containing the Laplacian, we deduce
\begin{equation}\label{1218-1}
\|\eqref{v-duhamel3}\|_{\ell_z^q L_{t,x}^q} \lesssim \biggl[\frac{T_0^{\frac16}}{\eps R^{\frac12}} + \frac{T_0^{\frac16}}{\eps^2 R^{\frac32}}\biggr]E_0^{\frac12}. 
\end{equation}

Recalling \eqref{nu} and \eqref{vz2} and collecting \eqref{1218-3}, \eqref{1218-4}, and \eqref{1218-1}, we deduce
\[
\|u_L\|_{L_{t,x}^{\frac{2(d+2)}{d-1}}([t_1,t_1+T_0^{\frac13}]\times\R^d)}  \lesssim \biggl[\biggl(\frac{R_0 e^J}{T_0 J \delta}\biggr)^{\frac14}+\frac{1}{\eps^{\frac12} R_0^{\frac12}}+\frac{T_0^{\frac16}}{\eps R_0^{\frac12}}+\frac{T_0^{\frac16}}{\eps^2 R_0^{\frac32}}\biggr] E_0^{\frac12}.
\]
Choosing $R_0$ and $T_0$ large enough (e.g. taking $T_0\sim R_0^{\frac37}[\tfrac{e^J}{J}]^{\frac3{14}}\delta^{-\frac3{14}}\eps^{\frac{6}{7}}$) and recalling the definition of $u_L$ yields
\[
\|e^{i(t-t_1)\Delta}u(t_1)\|_{L_{t,x}^{\frac{2(d+2)}{d-1}}([t_1,t_1+T_0^{\frac13}]\times\R^d)} \lesssim \eps E_0^{\frac12}.
\] 
Thus, for $\eps=\eps(E_0)$ small enough, a standard continuity argument yields
\[
\|u\|_{L_{t,x}^{\frac{2(d+2)}{d-1}}([t_1,t_1+T_0^{\frac13}]\times\R^d)} \lesssim \eps^{\frac12}.
\]
As $\eps>0$ was arbitrary, this implies \eqref{sc} with $t_0=t_1+T_0^{\frac13}\in(a,a+T_0).$  Appealing to Theorem~\ref{T:SC}, we complete the proof of Theorem~\ref{T}.\end{proof}

\end{document}